\documentclass[english,reqno]{amsart}
\usepackage[T1]{fontenc}
\usepackage[latin9]{inputenc}
\usepackage{amsthm}
\usepackage{amssymb}

\makeatletter
\numberwithin{equation}{section}
\numberwithin{figure}{section}
\theoremstyle{plain}
\newtheorem{thm}{\protect\theoremname}[section]
  \theoremstyle{plain}
  \newtheorem{cor}[thm]{\protect\corollaryname}
  \theoremstyle{plain}
  \newtheorem{lem}[thm]{\protect\lemmaname}
  \theoremstyle{remark}
  \newtheorem{rem}[thm]{\protect\remarkname}

\usepackage{amscd}

\theoremstyle{plain}

\newcommand{\N}{\mathbb{N}}
\newcommand{\R}{\mathbb{R}}

\makeatother

\usepackage{babel}
  \providecommand{\corollaryname}{Corollary}
  \providecommand{\lemmaname}{Lemma}
  \providecommand{\remarkname}{Remark}
\providecommand{\theoremname}{Theorem}

\begin{document}

\title{On the Brezis-Lieb Lemma without pointwise convergence}

\author{Adimurthi}

\address{TIFR CAM, Sharadanagar, P.B. 6503 Bangalore 560065, India}

\email{aditi@math.tifrbng.res.in}

\author{Cyril Tintarev}

\address{Uppsala University, P.O.Box 480, 75 106 Uppsala, Sweden }

\email{tintarev@math.uu.se}

\maketitle

\section{Introduction}

Brezis-Lieb Lemma (\cite{BL}) is a refinement of Fatou lemma that
plays an important role in analysis of partial differential equations.
Let $\Omega,\mu$ be a measure space. The lemma says that if $p\in[1,\infty)$,
$u_{k}\rightharpoonup u$ in $L^{p}(\Omega,\mu)$ and $u_{k}\to u$
a.e., then
\begin{equation}
\int_{\Omega}|u_{k}|^{p}d\mu-\int_{\Omega}|u|^{p}d\mu-\int_{\Omega}|u_{k}-u|^{p}d\mu\to0.\label{eq:BLorig}
\end{equation}
In concrete applications convergence a.e. might be hard to verify,
while the weak convergence condition rarely presents a difficulty,
since $L^{p}(\Omega,\mu)$ with $p\in(1,\infty)$ is reflexive and
any bounded sequence there has a weakly convergent subsequence. Thus
it is natural to ask what possible analogs of (\ref{eq:BLorig}) may
exist for sequences in $L^{p}$ that do not necessarily converge everywhere. This situation arises in applications to quasilinear elliptic PDE when $u_k$ are vector-valued functions of the form $\nabla w_k\in L^p$ and one cannot rely on compactness of local Sobolev imbeddings that yield a.e. convergence of $w_k$ but not of their gradients.
An immediate analog is given by weak semicontinuity of the norm, namely
\[
u_{k}\rightharpoonup u\:\Longrightarrow\int_{\Omega}|u_{k}|^{p}d\mu\ge\int_{\Omega}|u|^{p}d\mu+o(1),
\]
but this inequality is quite crude as it does not account for the
norm of the remainder $u_{k}-u$. 

On the other hand, there are some cases where Brezis-Lieb lemma holds
under assumption of weak convergence alone. One is when $\Omega$
is a countable set equipped with the counting measure, because in
this case pointwise convergence follows from weak convergence. Another
is the case $p=2$, when the conclusion of Brezis-Lieb lemma holds
even if convergence a.e. is not assumed. This follows from an elementary
relation in the general Hilbert space:

\begin{equation}
u_{k}\rightharpoonup u\:\Longrightarrow\|u_{k}\|^{2}=\|u_{k}-u\|^{2}+\|u\|-2(u_k-u,u)=\|u_{k}-u\|^{2}+\|u\|+o(1).\label{eq:HilbBl}
\end{equation}
Since in both examples the norm satisfies the Opial condition \cite{Opial},
it would be tempting to conjecture that the condition of a.e. convergence
may be dropped whenever the Opial condition holds, or, in case of
a strictly convex Banach space $X$ with single-valued duality map,
whenever the following sharp sufficient condition, which implies Opial
condition (see \cite{Opial}), holds: $u_{k}\rightharpoonup0$ in $X$ $\Longrightarrow u_{k}^{*}\rightharpoonup0$
. This prompted the authors of a forthcoming paper \cite{SoliTi} to
prove the following analog of Brezis-Lieb Lemma with a.e. convergence
replaced by weak convergence of a dual sequence. However, as we show
in Corollary \ref{cor:counterexample} below, the condition $p\ge3$
(that has nothing to do with Opial's condition or dual mapping) cannot
be relaxed. The condition  $|u_{k}-u|^{p-2}(u_{k}-u)\rightharpoonup0$ below is not arbitrary, but is an assumption of weak convergence of the duality mapping, which can be equivalently expressed as $(u_k-u)^*\rightharpoonup 0$. 
\begin{thm}
\label{thm:newbl}Let $(\Omega,\mu)$ be a measure space and let If
$p\in[3,\infty)$. Assume that $u_{k}\rightharpoonup u$ in $L^{p}(\Omega,\mu)$
and $|u_{k}-u|^{p-2}(u_{k}-u)\rightharpoonup0$ in $L^{p'}(\Omega,\mu)$,
$p'=\frac{p}{p-1}$. Then

\begin{equation}
\int_{\Omega}|u_{k}|^{p}d\mu\ge\int_{\Omega}|u|^{p}d\mu+\int_{\Omega}|u_{k}-u|^{p}d\mu+o(1).\label{eq:BL}
\end{equation}

\end{thm}
The proof of the theorem follows immediately from the following elementary
inequality, verified in \cite{SoliTi},

\begin{equation}key-1
|1+t|^{p}\ge1+|t|^{p}+p|t|^{p-2}t+pt,\mbox{\;|t|\ensuremath{\le}1},\label{eq:elem}
\end{equation}
which in turn implies $|u_{k}|^{p}\ge|u_{k}-u|^{p}+|u|^{p}+p|u|^{p-2}u(u_{k}-u)++p|u_{k}-u|^{p-2}(u_{k}-u)u$,
with the integrals of the last two terms vanishing by assumption.
Remarkably, (\ref{eq:elem}) is false for all $p\in(1,3)$, but this
does not imply that (\ref{eq:BL}) is false for these $p$, moreover,
as we mentioned above, it is true in the case of $\ell^{p}$, although
as we show in this note, it is false for $L^{p}([0,1])$. The inequality
in (\ref{eq:BL}) can be strict. Indeed, one can easily calculate
by binomial expansion for $p=4$ that if $u_{k}\rightharpoonup u$
and $(u_{k}-u)^{3}\rightharpoonup0$ in $L^{4/3}$, then

\[
\int_{\Omega}|u_{k}|^{4}d\mu=\int_{\Omega}|u|^{4}d\mu+\int_{\Omega}|u_{k}-u|^{4}d\mu+6\int u^{2}(u_{k}-u)^{2}d\mu+o(1).
\]
There have been some modifications of Brezis-Lieb lemma, in literature,
namely \cite{MT,MT2}, but we could not find any related results without
the assumption of the a.e. convergence. In this note we prove a generalization
of (\ref{eq:BL}) to the case of vector-valued functions and $p\ge3$,
and show in Corollary \ref{cor:counterexample}  that the inequality
(\ref{eq:BL}) is false for all $p\in(1,3)$. Other results in this
note are: a different weak convergence condition that yields (\ref{eq:BL})
for all $p\ge2$ (Theorem \ref{thm:newbl-1-1}), a version of Theorem
\ref{thm:newbl} for vector-valued functions (Theorem \ref{thm:newbl-1}),
and the analysis, in Section 3, of weak limits for sequences of the
form $\varphi\circ v_{k}$ with different functions $\varphi$.

\section{Theorem \ref{thm:newbl} for vector-valued functions}
\begin{thm}
\label{thm:newbl-1}Let $(\Omega,\mu)$ be a measure space and let
$p\in[3,\infty)$ and $m\in\N$. Assume that $u_{k}\rightharpoonup u$
in $L^{p}(\Omega,\mu;\R^{m})$ and $|u_{k}-u|^{p-2}(u_{k}-u)\rightharpoonup0$
in $L^{p'}(\Omega,\mu;\R^{m})$, $p'=\frac{p}{p-1}$. Then

\begin{equation}
\int_{\Omega}|u_{k}|^{p}d\mu\ge\int_{\Omega}|u|^{p}d\mu+\int_{\Omega}|u_{k}-u|^{p}d\mu+o(1).\label{eq:BL-1}
\end{equation}
\end{thm}
\begin{proof}
Once we prove the inequality
\begin{equation}
F(t,\theta):=|1+t^{2}+2t\theta|^{p/2}-1-|t|^{p}-p|t|^{p-2}t\theta-pt\theta\ge0,\mbox{\;|t|\ensuremath{\le}1},|\theta|\le1,\label{eq:elem-1}
\end{equation}
the assertion of the theorem will follow similarly to that of Theorem
\ref{thm:newbl}.

Note that for each $t\in[-1,1]$, the function $\theta\mapsto F(t,\theta)$
is convex on $[-1,1]$. An elementary computation shows that, for
any $t\in[-1,1]$, $\frac{\partial F(t,\theta)}{\partial\theta}\neq0$,
and thus $F(t,\theta)\ge\min\{F(t,-1),F(t,1)\}$. Since $F(t,-1)=F(-t,1)$
it suffices to show that $F(t,1)\ge0$ for all $t\in[-1,1].$ This
inequality, however, is nothing but (\ref{eq:elem}).
\end{proof}
Writing the statement of Theorem \ref{thm:newbl-1} in terms of gradients
of functions, and noting that $|\nabla u_{k}-\nabla u|^{p-2}(\nabla u_{k}-\nabla u)\rightharpoonup0$
in $L^{p'}(\Omega;\R^{N})$ can be rewritten in terms of the $p$-Laplacian,
as $-\Delta_{p}(u_{k}-u)\rightharpoonup0$ in the sense of distributions
(the relevant norm bound is already given as the $L^{p}$ bound for
the gradient in the first condition), we have
\begin{cor}
\label{cor:bl-grad}Let $\Omega\in\R^{N}$, $N\in\N$, be an open
set and let If $p\in[3,\infty)$. Assume that $\nabla u_{k}\rightharpoonup\nabla u$
in $L^{p}(\Omega;\R^{N})$ and $-\Delta_{p}(u_{k}-u)\rightharpoonup0$
in the sense of distributions. Then
\[
\int_{\Omega}|\nabla u_{k}|^{p}dx\ge\int_{\Omega}|\nabla u|^{p}dx+\int_{\Omega}|\nabla u_{k}-\nabla u|^{p}dx+o(1).
\]

\end{cor}

\section{weak convergence of compositions.}

Let $p\in(1,\infty)$. It is possible to construct a sequence $v_{k}\rightharpoonup0$
in $L^{p}([0,1])$ such that $|v_{k}|^{q-1}v_{k}$ has a nonzero weak
limit in $L^{p/q}([0,1])$ for any $q\in(1,p]$. We consider here
a more general case, comparing weak limits of sequences of the form
$\varphi(v_{k})$ with different odd continuous functions $\varphi$.

We focus here only on the measure space $[0,1]$ equipped with the
Lebesgue measure, but the argument can be easily adapted to domains
in $\R^{N}$. Let $T_{j}v(x)=v(jx)$ for $x\in[0,1/j]$, $j\in\N$,
extended periodically to the rest of the interval $[0,1]$. Note that
operators $T_{j}$ are isometries on $L^{p}([0,1])$. Oscillatory
sequences $T_{j}v$ always converge weakly to a constant function
as indicated in the following statement.
\begin{lem}
\label{lem:osc}If $v\in L^{p}([0,1]),$ $p\in(1,\infty),$ then $T_{j}v\rightharpoonup\int_{[0,1]}v\, dx$
in \textup{$L^{p}([0,1])$}.\end{lem}
\begin{proof}
Since $\|T_{j}v\|_{p}=\|v\|_{p}$, it suffices to verify that $\int T_{j}v\psi\; dx\to\int_{[0,1]}v\, dx\int_{[0,1]}\psi\, dx$
for all step functions $\psi$, since they form a dense subspace of
$L^{p'}([0,1])$. This, however, easily follows from a particular
case $\psi=1$, which in turn can be handled by applying periodicity
and rescaling of the integration variable. \end{proof}
\begin{lem}
\label{lem:osc2}Let $1< q\le p<\infty$. If $\varphi$ is a continuous
real-valued function on $\R$ such that for some $C>0$, $|\varphi(t)|\le C(1+|t|^{q})$,
and $v\in L^{p}([0,1])$, then $\varphi(T_{j}v)=T_{j}\varphi(v)\rightharpoonup\int_{[0,1]}\varphi(v(s))ds$$ $
in $L^{p/q}([0,1])$.\end{lem}
\begin{proof}
Let $v$ be first a step function with values $t_{j}$ on intervals
of length $m_{j}$, $j=1,\dots,M$. By Lemma \ref{lem:osc}, $\varphi(T_{k}v)\rightharpoonup\sum_{j}\varphi(t_{j})m_{j}=0$.
The assertion of the lemma follows then from density of step functions
in $L^{p}$. \end{proof}
\begin{thm}
\label{thm:linindep}Let $\varphi_{i}$, $i=1,\dots,M$, be continuous
functions $\R\to\R$, odd for each $i\neq M$, and assume that for
some $q\ge1$, $C>0$, $|\varphi_{i}(t)|\le C(1+|t|^{q})$, $i=1,\dots,M$.
If for every sequence $v_{k}\in L^{\infty}([0,1])$, such that $\varphi_{i}(v_{k})\rightharpoonup0$
in $L^{1}([0,1])$, $i=1,\dots M-1$, one also has $\varphi_{M}(v_{k})\rightharpoonup0$
in $L^{1}([0,1])$, then the functions $\{\varphi_{i}\}_{i=1,\dots M}$
are linearly dependent.\end{thm}
\begin{proof}
Let $\psi\ge1$ be a Lipschitz continuous function on $[-a,a]\subset\mathbb{R}$,
$a>0$, and let $v$ be a solution of the equation 
\[
v'(t)=\frac{\gamma}{\psi(v(t))},\, v(0)=-a,
\]
with the value of $\gamma=\gamma(\psi)>0$ set to satisfy $v(1)=a$.
Such $\gamma$ always exsists, since $v'(t)\le\gamma$ and thus $v(1)\le-a+\gamma$,
and on the other hand, $v(1)\ge-a+\frac{\gamma}{\psi(-a)+L(v(1)+a)}$,
where $L$ is the Lipschitz constant of $\psi$, and thus $v(1)$
is a continuous function of $\gamma\in(0,\infty)$ with the range
$(-a,+\infty)$. 

By Lemma \ref{lem:osc2},
\begin{equation}
\varphi_{i}(T_{k}v)\rightharpoonup\int_{[0,1]}\varphi_{i}(v(s))\, ds=\gamma^{-1}\int_{[-a,a]}\varphi_{i}(t)\psi(t)\, dt,\label{eq:wlim}
\end{equation}
 with the weak convergence in $L^{p}([0,1])$ for any $p\ge1$.

Consider now a closure $Y$ in $L^{2}([-a,a])$ of the span of all
positive bounded continuous functions $\psi$ on $[-a,a]$, such that
$(\varphi_{i},\psi)_{L^{2}([-a,a])}=0$, $i=1,\dots,M-1$. Note that
$Y$contains all positive even functions and thus is nontrivial. Furthermore,
$Y$ is the orthogonal complement of $\{\varphi_{i}\}_{i=1,\dots,M-1}$
in $L^{2}$: indeed, any function can be approximated by a bounded
function in this orthogonal complement, and adding a large constant
to the latter makes it a positive function orthogonal to $\{\varphi_{i}\}_{i=1,\dots,M-1}$.
By assumption, it follows from (\ref{eq:wlim}) that $\varphi_{M}\perp Y$,
and consequently it belongs to the span of $\varphi_{1},\dots,\varphi_{M-1}$
as functions on $[-a,a]$. Since the value of $a>0$ is arbitrary,
on may conclude (assuming without loss of generality that $\varphi_{1},\dots,\varphi_{M-1}$
are linearly independent, so that the coefficients in expansion of
$\varphi_{M}$ as a linear combination of $\varphi_{1},\dots,\varphi_{M-1}$
are unique), the functions $\varphi_{1},\dots,\varphi_{M}$ are linearly
dependent also as functions on $\R$. \end{proof}
\begin{cor}
\label{cor:nonzero}Let $\varphi_{i}$, $i=1,\dots,M$, be continuous
linearly independent nonzero functions $\R\to\R$, odd for each $i\neq M$,
and assume that for some $q\ge1$, $C>0$, $|\varphi_{i}(t)|\le C(1+|t|^{q})$,
$i=1,\dots,M$. There exists a sequence $v_{k}\in L^{\infty}([0,1])$,
such that $\varphi_{i}(v_{k})\rightharpoonup0$ in $L^{1}([0,1])$,
$i=1,\dots M-1$, while there is $\alpha\neq0$ such that $\varphi_{M}(v_{k})\rightharpoonup\alpha$.
If the functions $\varphi_{i}$, $i=1,\dots,M$, are piecewise-$C^{1}$
and linearly independet on any interval, and $\varphi_{M}$ changes
sign, the sequence $v_{k}$ can be chosen so that $\alpha<0$.\end{cor}
\begin{proof}
The first assertion of the corollary is immediate from Theorem \ref{thm:linindep}.
Assume now, in view of Lemma \ref{lem:osc2}, that for every $v\in L^{\infty}([0,1])$,
such that $\varphi_{i}(T_{k}v)\rightharpoonup\int_{[0,1]}\varphi_{i}(v(s))ds=0$,
$i=1,\dots,M-1$, we have $\alpha=\int_{[0,1]}\varphi_{M}(v(s))ds\ge0$.
We have therefore that 
\begin{equation}
\inf_{\int_{[0,1]}\varphi_{i}(v(s))ds=0,\; i=1,\dots,M-1}\int_{[0,1]}\varphi_{M}(v(s))ds=0.\label{eq:inf}
\end{equation}
It is easy to show that there exists a non-zero bounded function $v_{0}$
such that $\int_{[0,1]}\varphi_{M}(v_{0}(s))ds=0$. Indeed, let $a,b\in\R$
be such that $\varphi_{M}(a)<0<\varphi_{M}$(b). By continuity of
$\varphi_{M}$ there exist an $\epsilon>0$ such that for any functions
$v$ and $w$ such that $\|v-a\|_{\infty}<\epsilon$ and $\|w-b\|_{\infty}<\epsilon$,
one has $\varphi_{M}(v)<0$ and $\varphi_{M}(w)>0$. Fix any such
$v,w\in C^{1}$ whose derivatives are linearly independent. Then the
function $\theta\mapsto\int_{[0,1]}\varphi_{M}(\theta v+(1-\theta)w)$,
$0\le\theta\le1$, will change sign and thus it will vanish at some
$\theta_{0}\in(0,1)$ by the intermediate value theorem. The function
$v_{0}=\theta_{0}v+(1-\theta_{0})w$ will not be a constant by the
assumption of linear independence.

Then $v_{0}$ is a point of minimum in \ref{eq:inf}, and by the Lagrange
multiplier rule, there exsist real numbers $\lambda_{1},\dots,\lambda_{M-1}$
such that for any $t$ in the range of $v_{0}$ where the functions
$\varphi_{i}$ are differentiable, 
\[
\varphi'_{M}(t)=\lambda_{1}\varphi'_{1}(t)+\dots+\lambda_{M'1}\varphi'_{M-1}(t).
\]
Since functions $\{\varphi_{i}\}_{i=1,\dots,M}$ are linearly independent
on any interval and are piecewise differentiable, we have a contradiction.\end{proof}
\begin{cor}
\label{cor:counterexample}Let $\Omega=[0,1]$, equipped with the
Lebesgue measure. Then for any $p\in[1,3)$$ $, there exists a sequence
$v_{k}\in L^{\infty}([0,1])$ such that $v_{k}\rightharpoonup0$ in
$L^{p}$, $|v_{k}|^{p-2}v_{k}\rightharpoonup0$ in $L^{p'}([0,1])$,
but the relation (\ref{eq:BL}) with $u_{k}=1+v_{k}$ does not hold.\end{cor}
\begin{proof}
Let $F_{p}(t)=|1+t|^{p}-1-|t|^{p}$. Given $ $$ $$1\le p<3$ , the
function $F_{p}$ changes sign. Apply Corollary \ref{cor:nonzero}
with $M=3$, $\varphi_{1}(t)=t$ and $\varphi_{2}(t)=|t|^{p-2}t$
and $\varphi_{3}(t)=F_{p}(t)$. \end{proof}
\begin{rem}
Note that this counterexample cannot be extended to all measure spaces,
since, as we have noted, (\ref{eq:BLorig}) holds in $\ell^{p}$ under
the assumption of weak convergence alone.
\end{rem}

\section{A version of Brezis-Lieb lemma.}

In the previous section we observed, roughly speaking, that weak limits
of $\varphi_{i}(u_{k})$ for linearly independent functions $\varphi_{i}$
have independent values, and that the inequality $\int_{[0,1]}\varphi_{M}(v_{k})\ge o(1)$
holds for all sequences satisfying $\varphi_{i}(v_{k})\rightharpoonup0$,
$i=1,\dots,M$, only if $\varphi_{M}(t)-\sum_{i=1}^{M-1}\lambda_{i}\varphi_{i}(t)\ge0$
for some real $\lambda_{1},\dots,\lambda_{M}$. Therefore one may
as well use the condition $\Phi(v_{k})\rightharpoonup0$ with
$\Phi(t)=\sum_{i=1}^{M-1}\lambda_{i}\varphi_{i}(t)$. In particular,
the function $F_{p}(t)=|1+t|^{p}-1-|t|^{p}$ , $p\ge2$, dominates
the following function: $\Phi(t)=pt$ for $|t|\le1$, $\Phi(t)=p|t|^{p-2}t$
for $|t|>1$.
\begin{thm}
Let \label{thm:newbl-1-1}Let $(\Omega,\mu)$ be a measure space and
let If $p\ge2$. Assume that $u_{k}\in L^{p}(\Omega,\mu)$, $u\in L^{p}(\Omega,\mu)$
and $\Psi(u,u_{k}-u)\rightharpoonup0$ in $L^{1}(\Omega,\mu)$, where

\[
\Psi(s,t)=\begin{cases}
|s|^{p-1}t, & |t|\le|s|,\\
|s||t|^{p-2}t,\: & |t|\ge|s|
\end{cases}\:.
\]
Then

\begin{equation}
\int_{\Omega}|u_{k}|^{p}d\mu\ge\int_{\Omega}|u|^{p}d\mu+\int_{\Omega}|u_{k}-u|^{p}d\mu+o(1).\label{eq:BL-1-1}
\end{equation}
\end{thm}
\begin{proof}
This follows from the inequality $F_{p}(\lambda)\ge\Phi(\lambda),$
from which, with $\lambda=\frac{u_{k}(x)-u(x)}{u(x)}$, whenever $u(x)\neq 0$, 
immediately follows
\[
|u_{k}|^{p}-|u|^{p}-|u-u_{k}|^{p}\ge\Psi(u,u_{k}-u).
\]
\end{proof}

\end{document}